\documentclass[a4paper,reqno, 11pt]{amsart}  
\usepackage[margin=1in]{geometry}
\usepackage[centertags]{amsmath}
\usepackage{amstext,amssymb,amsopn,amsthm}
\usepackage{nicefrac, esint}
\usepackage{dsfont}
\usepackage{thmtools}
\usepackage{graphicx}
\usepackage{hyperref}

\declaretheorem[name=Theorem, numberwithin=section]{theorem}
\newtheorem{corollary}[theorem]{Corollary}
\newtheorem{lemma}[theorem]{Lemma}
\newtheorem{proposition}[theorem]{Proposition}
\newtheorem{definition}[theorem]{Definition}
\newtheorem{assumption}[theorem]{Assumption}

\newtheorem*{example*}{Example}
\newtheorem*{robremark*}{Robustness remark}

\declaretheoremstyle[bodyfont=\normalfont]{remark-style}
\declaretheorem[name={Remark}, style=remark-style, unnumbered]{remark}

\numberwithin{equation}{section}

\theoremstyle{plain}

\newcommand{\N}{\mathds{N}}
\newcommand{\R}{\mathds{R}}

\newcommand{\BIGOP}[1]
{
\mathop{\mathchoice%
{\raise-0.22em\hbox{\huge $#1$}}%
{\raise-0.05em\hbox{\Large $#1$}}{\hbox{\large $#1$}}{#1}}}

\def\XXint#1#2#3{{\setbox0=\hbox{$#1{#2#3}{\int}$}
     \vcenter{\hbox{$#2#3$}}\kern-.5\wd0}}

\newcommand{\BIGboxplus}{\mathop{\mathchoice%
{\raise-0.35em\hbox{\huge $\boxplus$}}%
{\raise-0.15em\hbox{\Large $\boxplus$}}{\hbox{\large $\boxplus$}}{\boxplus}}}

\DeclareMathOperator{\dist}{dist}

\DeclareMathOperator{\radius}{radius}
\DeclareMathOperator{\cent}{center}

\renewcommand{\d}{\textnormal{d}}

\title{Coercivity estimates for integro-differential operators}
\date{\today}
\thanks{Luis Silvestre is supported in part by NSF grant DMS-1764285. Jamil Chaker is supported by DFG Forschungsstipendium through Project 410407063.}

\author{Jamil Chaker}
\address{(\textnormal{J. Chaker}), Mathematics Department, University of Chicago, Chicago, Illinois 60637, USA}
\email{jchaker@math.uchicago.edu}

\author{Luis Silvestre}
\address{(\textnormal{L. Silvestre}), Mathematics Department, University of Chicago, Chicago, Illinois 60637, USA}
\email{luis@math.uchicago.edu}

\begin{document}
\begin{abstract}
We provide a general condition on the kernel of an integro-differential operator so that its associated quadratic form satisfies a coercivity estimate with respect to the $H^s$-seminorm.
\end{abstract}
\maketitle
\section{Introduction}

In this article, we are interested in coercivity estimates for integro-differential quadratic forms in terms of fractional Sobolev norms. More precisely, we seek general conditions on a kernel $K(x,y)$ so that the following inequality holds for some constant $c>0$ and any function $u \in H^s$,
\begin{equation} \label{e:coercivity}
 \iint_{\R^d \times \R^d} |u(x)-u(y)|^2 K(x,y) \d x \d y \geq c \|u\|_{\dot H^s}^2.
\end{equation}
Here, $\dot H^s$ refers to the homogeneous fractional Sobolev norm whose standard expression is given by
\[ \|u\|_{\dot H^s}^2 = c_{d,s} \iint_{\R^d \times \R^d} |u(x)- u(y)|^2 |x-y|^{-d-2s} \d x \d y = \int_{\R^d} |\hat u(\xi)|^2 |\xi|^{2s} \d \xi .\]

The quadratic form is naturally associated with the linear integro-differential operator
\begin{equation} \label{e:integro-differential_operator}
 Lu (x) = PV \int_{\R^d} (u(y) - u(x)) K(x,y) \d y.
\end{equation}

Equations involving integro-differential diffusion like \eqref{e:integro-differential_operator} have been the subject of intensive research in recent years. The understanding of the analog of the theorem of De Giorgi, Nash and Moser in the integro-differentiable setting plays a central role in the regularity of nonlinear integro-differential equations (See \cite{komatsu1995}, \cite{basslevin2002}, \cite{barlow2009non}, \cite{kassmann2009priori}, \cite{MR3169755}, \cite{chan2011}, \cite{kassmann2013regularity}, \cite{Dyda-Kassmann-2015}, \cite{Imbert-Silvestre-2016} and references therein). It concerns the generation of a H\"older continuity estimate for solutions of parabolic equations of the form $u_t = Lu$, with potentially very irregular kernels $K$. There are diverse results in this direction with varying assumptions on $K$. The two key conditions that are necessary for this type of results are the coercivity condition \eqref{e:coercivity} and the boundedness of the corresponding bilinear form:
\begin{equation} \label{e:bilinear_bounded}
 \iint (u(y) - u(x))(v(y)-v(x)) K(x,y) \d x \d y \leq C \|u\|_{H^s} \|v\|_{H^s}.
\end{equation}
The initial works in the subject (like \cite{komatsu1995}, \cite{basslevin2002} or \cite{chan2011}) were focusing on kernels satisfying the convenient point-wise non-degeneracy assumption $\lambda |x-y|^{-d-2s} \leq K(x,y) \leq \Lambda |x-y|^{-d-2s}$. These two inequalities easily imply \eqref{e:coercivity} and \eqref{e:bilinear_bounded}. However, \eqref{e:coercivity} and \eqref{e:bilinear_bounded} hold under much more general assumptions. In \cite{kassmann2013regularity} and \cite{Dyda-Kassmann-2015}, the coercivity estimate \eqref{e:coercivity} is an assumption of the main theorem and some examples are given where the estimate applies to degenerate kernels. There are also recent applications of this framework to the Boltzmann equation (See \cite{Imbert-Silvestre-2016}) where the kernels are not point-wise comparable to $|x-y|^{-d-2s}$ and yet \eqref{e:coercivity} and \eqref{e:bilinear_bounded} hold.

While we know a fairly satisfactory general condition that ensures \eqref{e:bilinear_bounded} (See Section 4.1 in \cite{Imbert-Silvestre-2016}), assumptions that would ensure \eqref{e:coercivity} are not well understood. Simple examples of the form  $K(x,y) = b((x-y)/|x-y|) |x-y|^{-d-2s}$ can be analyzed using Fourier analysis (See \cite{ros2016regularity}) and they suggest that a condition that implies \eqref{e:coercivity} might be that for any point $x$, $r>0$ and any unit vector $e \in S^{d-1}$, we have
\begin{equation} \label{e:conjecture}
 \int_{B_r(x)} ((y-x) \cdot e)_+^2 K(x,y) \d y \geq \lambda r^{2-2s}.
\end{equation}
In \cite{Dyda-Kassmann-2015}, it is conjectured that \eqref{e:conjecture} implies \eqref{e:coercivity}. That conjecture is also mentioned in \cite{Imbert-Silvestre-2016}. We are not yet able to determine whether \eqref{e:conjecture} is sufficient to ensure that \eqref{e:coercivity} holds. We make the following assumption on the kernel. Essentially, it says that from every point $x$, the nondegeneracy set $\{y : K(x,y) \gtrsim |x-y|^{-d-2s}\}$ has some density in all directions.

\begin{assumption}\label{assumption}
There is $\mu\in(0,1)$ and $\lambda>0$ such that for every ball $B\subset\R^d$ and $x\in B$:
\begin{equation}\label{eqassum}
 |\{z\in B\colon K(x,z) \geq \lambda |x-z|^{-d-2s}\}| \geq \mu |B|. \tag{A1}
\end{equation}
\end{assumption}

\begin{remark} 
Note that we aim to prove estimates for energy forms and sets of measure zero can be neglected for integration.
Hence, \autoref{assumption} could be effortlessly relaxed by assuming the property \eqref{eqassum} for almost every $x\in B$ instead of every $x\in B$.
\end{remark}

We now state our main results.
\begin{theorem}\label{thm:coercive}
Assume there exist $\lambda>0$ and $\mu\in(0,1)$ such that the kernel $K$ satisfies \autoref{assumption}. 
There is a constant $c>0$, depending on the dimension $d$ and $\mu$ only, such that for every $u:\R^d\to\R$,
\[ \int_{\R^d}\int_{\R^d} (u(x)-u(y))^2\, K(x,y) \, \d y \, \d x \geq c\lambda \|u\|_{\dot{H}^s(\R^d)}^2. \]
\end{theorem}
Our second main result is a localized version of Theorem \ref{thm:coercive}. 
Indeed the approach we use in the proof of \autoref{thm:coercive} allows us to prove a localized lower bound estimate with some minor additional work.
\begin{theorem}\label{thm:local}
Assume there exist $\lambda>0$ and $\mu\in(0,1)$ such that $K$ satisfies \autoref{assumption}.
There is a constant $c>0$, depending on the dimension $d$ and $\mu$ only, such that for every function $u:\R^d\to\R$
\[ \int_{B_2}\int_{B_2} (u(x)-u(y))^2\, K(x,y)\, \d y\, \d x \geq c\lambda \|u\|_{\dot{H}^s (B_1)}^2. \]
\end{theorem}

Here, $\|u\|_{\dot H^s(B_1)}$ stands for Gagliardo's seminorm
\[ \|u\|_{\dot{H}^s (B_1)}^2 = \int_{B_1} \int_{B_1} \frac{|u(x) - u(y)|^2}{|x-y|^{d+2s}} \d x \, \d y.\]

The purpose of our theorems is to provide a criteria to verify the coercivity estimate \eqref{e:coercivity} based on a general condition on the kernel $K$ that is easy to verify in concrete examples. For example, coercivity estimates are known to hold for the non-cutoff Boltzmann collision operator  with parameters depending on hydrodynamic quantities. There is a long history of the derivations and use of these estimates. An early version with respect to a sub-optimal Sobolev exponent was obtained by P.L. Lions in \cite{LionsReg}. A sharp coercivity estimate appeared in the paper by Alexandre, Desvillettes, Villani and Wennberg \cite{AlexDesvVilWenn} which was proved using Fourier analysis. There is a simplified proof using Littlewood-Paley analysis in \cite{AlexandreLittlewood1} and \cite{AlexandreLittlewood2}. A proof based on a more geometrical argument (essentially measuring the intersection between two cones) is given in the appendix of \cite{Imbert-Silvestre-2016}. The precise asymptotic behavior of these coercivity estimates for large velocities is analyzed by Gressman and Strain in \cite{Gressmann-Strain-2011}. See also \cite{MouhotExpl},  \cite{AlexandreI}, \cite{AlexandreSome}, \cite{ChenHe2011}, \cite{he2016sharp}, \cite{AlexandreReview} and references therein. All the proofs in the literature use the specific structure of the Boltzmann collision operator, which is a nonlinear integro-differential operator. In \cite{SilvestreBoltzmann}, the Boltzmann collision operator is written in the form \eqref{e:integro-differential_operator} with a kernel $K$ that depends on the solution $f$ itself. Some basic properties of this kernel $K$ are easily observed from this computation. The coercivity estimate for the Boltzmann collision operator follows then as a direct application of Theorem \ref{thm:local} as a black box.

We now review some earlier works aiming at general conditions on a kernel $K$ to ensure the coercivity of the quadratic form \eqref{e:coercivity}. This is essentially the same objective as in this paper. In \cite{Dyda-Kassmann-2015}, they study kernels $K$ that satisfy $K(x,y) \approx k(x-y)$ for some fixed kernel $k$ that might contain a singular part. A binary operator $\heartsuit$ is defined for any such kernels $k$ that allows them to obtain an inequality like \eqref{e:coercivity} for some degenerate kernels.  Several examples are given. In \cite{Bux-Kassmann-Schulze-2017}, they study kernels such that $K(x,y) \geq \lambda |x-y|^{-d-2s}$ for every point $y$ in certain cone of directions centered at $x$. These cones are supposed to have a fixed opening, but might rotate arbitrarily from point to point. Our result in this paper implies the result in \cite{Bux-Kassmann-Schulze-2017}.

We now describe the outline of the proof in this paper. We build a sequence of kernels $K_j$ whose corresponding quadratic forms are smaller than the left hand side of \eqref{e:coercivity}. The basic mechanism for constructing these kernels is given in Lemma \ref{lemma:kernelsdecreasing}. Basically, it is an operation that given two kernels whose quadratic forms are bounded above, it produces a third kernel with the same upper bound. It is somewhat reminiscent to the $\heartsuit$ operator defined in \cite{Dyda-Kassmann-2015}, but it applies to more generic kernels $K(x,y)$ and allows us for more flexibility in the formula. We then analyze the nondegeneracy sets of these kernels $\mathcal N^j(x) := \{y : K(x,y) \geq a_j \lambda |x-y|^{-d-2s}\}$ for some sequence $a_j>0$. Using a covering argument similar to the \emph{growing ink spots lemma} by Krylov and Safonov \cite{krylov1980certain}, we prove that the density of these sets expands as $j$ increases. Moreover, it fills up the full space after finitely many iterations. Finally, we find a universal number $n \in \N$ so that $K_n(x,y) \geq a_n \lambda |x-y|^{-d-2s}$ for all pairs of points $x$ and $y$. The coercivity estimate \eqref{e:coercivity} follows from that.

As we said before, we aim at developing a theorem that is ready to be applied to obtain the coercivity estimate \eqref{e:coercivity} under the least restrictive assumptions possible. Predictably, the proof of Theorem \ref{thm:coercive} is not shorter than the proofs in the literature that apply to particular instances of kernels on a case by case basis. For example, the proof in the appendix of \cite{Imbert-Silvestre-2016} is quite a bit shorter than the proof in this paper. The reason is that the Boltzmann kernel has a special structure that, in the language of this paper, allows you to prove that $\mathcal N^1$ is already the full space $\R^d$ (thus, the proof  finishes after only one iteration).

There are some significant instances of kernels $K(x,y)$ that satisfy \eqref{e:coercivity} but are not covered by our Assumption \ref{assumption}. The main example is when $K(x,y) \d x \d y$ is actually a singular measure. That is the case in Example 4 in \cite{Dyda-Kassmann-2015}. In the context of the Boltzmann equation, the collision kernel would satisfy Assumption \ref{assumption} in terms of the mass, energy and entropy densities (this follows directly from the formulas in \cite{SilvestreBoltzmann}). However, if we replace the upper bound on the entropy density by a bound from below on the temperature tensor, the Boltzmann collision kernel would satisfy \eqref{e:conjecture} but not Assumption \ref{assumption}. In particular our Theorem \ref{thm:local} would suffice to imply Corollary L but not Theorem 1 in \cite{Gressmann-Strain-2011}.

We finish the introduction by describing the outline of the article. In Section \ref{s:kernels} we describe the construction of the sequence of kernels $K_j$. In Section \ref{s:nondeg-sets}, we analyze their corresponding sets of nondegeneracy. In Section \ref{s:theorems} we finish the proofs of our main theorems, including a covering argument that is necessary for the proof of Theorem \ref{thm:local}.

\section{Preliminaries}
\subsection{Notation}
We use the letter $c$ with subscripts for positive constants whose exact values are not important.\\
Let $C>0$. For a ball $B=B_r(x)$, we denote by $CB$ the scaled ball $CB=B_{Cr}(x)$.
\subsection{Reformulations of \autoref{assumption}}
This subsection is devoted to show that \autoref{assumption} can be reformulated in several equivalent ways which allows us to change the position of the 
point $x$ in the relation to the ball of consideration by modifying the value of $\mu$.
\begin{lemma}\label{lemm:ass}
The following statements are equivalent:
\begin{enumerate}
 \item[(A1)] There exist $\mu\in(0,1)$ and $\lambda>0$ such that $K$ satisfies \autoref{assumption}.
 \item[(A2)] There exist $\mu\in(0,1)$ and $\lambda>0$ such that for every ball $B\subset\R^d$ and $x\in \partial B$:
\[ |\{z\in B\colon K(x,z) \geq \lambda |x-y|^{-d-2s}\}| \geq \mu |B|. \]
 \item[(A3)] There exist $\mu\in(0,1)$, $c\in(0,1)$ and $\lambda>0$ such that for every ball $B_R(z_0)$ and $x\in\R^d$ with $|x-z_0|=(1+c)R$: 
\[ |\{z\in B_R(z_0) \colon K(x,z) \geq \lambda |x-y|^{-d-2s}\}|\geq \mu |B_R|. \]
\end{enumerate}
\end{lemma}
\begin{proof} 
(A1)$\Rightarrow$(A2):  Let $x\in\R^d$ and $B_R(z_0)$ a ball such that $x\in\partial B_R(z_0)$. Let $\epsilon>0$. By (A1), there exist $\mu\in(0,1)$ and $\lambda>0$ such that
\[ |\{z\in B_{R+\epsilon}(z_0)\colon K(x,z) \geq \lambda |x-y|^{-d-2s}\}|  \geq \mu |B_{R+\epsilon}|. \]
By continuity, (A2) follows for $\epsilon\to 0$.

(A2)$\Rightarrow$(A1):  Let $x\in\R^d$ and $B_R(z_0)$ a ball such that $x\in B_R(z_0)$. There is a ball $B\subset B_{R}(z_0)$ with radius greater or equal $\frac12 R$ such that $x\in\partial B$. 
By (A2), there exist $\widetilde{\mu}\in(0,1)$ and $\lambda>0$ such that $ |\{z\in B\colon K(x,z) \geq \lambda |x-y|^{-d-2s}\}| \geq \widetilde{\mu}|B|$. Choosing $\mu=\widetilde{\mu}/2^d$, leads to
\[ |\{z\in B_R(z_0)\colon K(x,z) \geq \lambda |x-y|^{-d-2s}\}|  \geq |\{z\in B\colon K(x,z) \geq \lambda |x-y|^{-d-2s}\}|  \geq \widetilde{\mu} |B| \geq \mu |B_R|. \]

(A2)$\Rightarrow$(A3): Let $c\in(0,1)$, $x\in\R^d$ and $B_R(z_0)$ a ball such that $|x-z_0|=(1+c)R$. By (A2) there is $\widetilde{\mu}\in(0,1)$ and $\lambda>0$ such that
\[  |\{z\in B_{(1+c)R}(z_0) \colon K^j(x,z) \geq \lambda |x-y|^{-d-2s}\}|\geq \mu |B_{(1+c)R}| = \widetilde{\mu}(1+c)^d |B_{R}|. \]
Hence,
\begin{align*}
&  |\{z\in B_R(z_0) \colon K^j(x,z) \geq \lambda |x-y|^{-d-2s}\}| \\
 & \geq |\{z\in B_{(1+c)R}(z_0) \colon K^j(x,z) \geq \lambda |x-y|^{-d-2s}\}| -  |B_{(1+c)R}(z_0)\setminus B_{R}(z_0)|\\
 &\geq \widetilde{\mu}(1+c)^d |B_{R}| + |B_{(1+c)R}(z_0)\setminus B_{R}(z_0)| \\
 &= (\widetilde{\mu}(1+c)^d - (1+c)^d+1) |B_{R}|.
\end{align*}
Choosing 
\[ 0<c<\min\left(1,\left(\frac{1}{1-\widetilde{\mu}}\right)^{1/d}-1\right) \]
and $\mu:=(\widetilde{\mu}(1+c)^d - (1+c)^d+1)< 1$, proves (A3).

(A3)$\Rightarrow$(A2): Let $x\in\R^d$ and $B_R(z_0)$ a ball such that $x\in\partial B_R(z_0)$. 
By (A3) there is $\widetilde{\mu}\in(0,1)$, $c\in(0,1)$ and $\lambda>0$ such that
\[  |\{z\in B_{r}(z_0) \colon K^j(x,z) \geq \lambda |x-y|^{-d-2s}\}|\geq \mu |B_{(1+c)R}| = \widetilde{\mu}|B_{r}|, \]
where $r=R/(1+c)$. Hence, (A2) follows by choosing $\mu=\widetilde{\mu}2^{-d}$:
\begin{align*}
&  |\{z\in B_R(z_0) \colon K^j(x,z) \geq \lambda |x-y|^{-d-2s}\}| \geq |\{z\in B_{r}(z_0) \colon K^j(x,z) \geq \lambda |x-y|^{-d-2s}\}| \\
&\geq \widetilde{\mu}|B_r| = \widetilde{\mu}(1+c)^{-d} |B_R| \geq \widetilde{\mu} 2^{-d}|B_R|=\mu|B_R|.
\end{align*}
\end{proof}
\begin{remark}
It can be easily seen in the foregoing proof that the value of $\lambda$ does not change in the transition from one statement into the other.
Hence, the constant $\lambda>0$  can be chosen to be the same in all three statements in \autoref{lemm:ass}. 
\end{remark}
\section{Diffusing the kernels} \allowdisplaybreaks
\label{s:kernels}
In this section we introduce auxiliary kernels and corresponding sets of non-degeneracy. Furthermore, we establish some basic properties for these objects.

\begin{lemma}\label{lemma:kernelsdecreasing}
Assume $K,K_1, K_2:\R^d\times \R^d \to[0,\infty)$ are kernels such that for every $u:\R^d\to\R$
\begin{align*}
c_1 \int_{\R^d}\int_{\R^d} (u(x)-u(y))^2 K_1(x,y)\, \d x\, \d y \leq \int_{\R^d}\int_{\R^d} (u(x)-u(y))^2 K(x,y)\, \d x\, \d y, \\
c_2 \int_{\R^d}\int_{\R^d} (u(x)-u(y))^2 K_2(x,y)\, \d x\, \d y \leq \int_{\R^d}\int_{\R^d} (u(x)-u(y))^2 K(x,y)\, \d x\, \d y
\end{align*}
for some constants $c_1,c_2>0$.
Consider two functions $\eta_1,\eta_2:\R^d\times\R^d\times \R^d\to [0,\infty)$ such that
\begin{align*}
\int_{\R^d} \eta_1(x,y,z)\, \d y \leq 1 \quad \text{for all } x,z\in\R^d, \\
\int_{\R^d} \eta_2(x,y,z)\, \d x \leq 1 \quad \text{for all } y,z\in\R^d. 
\end{align*}
Then $K_3:\R^d\times \R^d\to [0,\infty)$,
\[K_3(x,y) : = \int_{\R^d} \min ( K_1(x,z)\eta_1(x,y,z),K_2(y,z)\eta_2(x,y,z)) \, \d z\]
also satisfies
 \[c_3 \int_{\R^d}\int_{\R^d} (u(x)-u(y))^2 K_3(x,y)\, \d x\, \d y \leq \int_{\R^d}\int_{\R^d} (u(x)-u(y))^2 K(x,y)\, \d x\, \d y\]
 for some constant $c_3>0$ depending on $c_1$ and $c_2$ only.
\end{lemma}
\begin{proof}
By Fubini's theorem and $2 |u(x)-u(z)|^2 + 2|u(y)-u(z)|^2 \geq |u(x)-u(y)|^2$,
\begin{align*}
& \int_{\R^d}\int_{\R^d} (u(x)-u(y))^2 K_3(x,y)\, \d x\, \d y \\
& =  \int_{\R^d}\int_{\R^d} (u(x)-u(y))^2 \int_{\R^d} \min ( K_1(x,z)\eta_1(x,y,z),K_2(y,z)\eta_2(x,y,z)) \, \d z\, \d x\, \d y \\
& \leq 2\int_{\R^d}\int_{\R^d} \int_{\R^d} (u(x)-u(z))^2 K_1(x,z)\eta_1(x,y,z) \, \d z\, \d x\, \d y \\
& \quad + 2\int_{\R^d}\int_{\R^d} \int_{\R^d} (u(y)-u(z))^2 K_2(y,z)\eta_2(x,y,z) \, \d z\, \d x\, \d y \\
& \leq 2\int_{\R^d}\int_{\R^d} (u(x)-u(z))^2 K_1(x,z) \, \d z\, \d x +  2\int_{\R^d} \int_{\R^d} (u(y)-u(z))^2 K_2(y,z) \, \d z \, \d y \\
& \leq 2\left(\frac{1}{c_1}+\frac{1}{c_2}\right) \int_{\R^d} \int_{\R^d} (u(x)-u(y))^2 K(x,y) \, \d x \, \d y.
\end{align*}
\end{proof}

We iteratively define sequences of auxiliary kernels. 
\begin{definition}\label{def:auxkernels}
Let $K^0:\R^d\times\R^d\to[0,\infty)$, $K^0(x,y):=K(x,y)$. 
We define for $j\geq 0$ the sequence of auxiliary kernels $K^{j+1}:\R^d\times\R^d\to [0,\infty)$ by
\[ K^{j+1}(x,y) : = \int_{\R^d} \min ( K^j(x,z)\eta^j_1(x,y,z),K(y,z)\eta_2(x,y,z)) \, \d z,\]
where $\eta^j_1,\eta_2:\R^d\times\R^d\times\R^d\to[0,\infty)$ are functions satisfying for all $x,z\in\R^d$ resp. $y,z\in\R^d$:
  \begin{equation}\label{eta12}
\int_{\R^d}\eta^j_1(x,y,z) \, \d y \leq 1 \quad \text{ and } \int_{\R^d}\eta_2(x,y,z)\, \d x \leq 1. 
 \end{equation}
\end{definition}
\begin{remark}
For the moment, the functions $\eta_1^j, \eta_2$ are generic functions satisfying \eqref{eta12}. 
The explicit form of those functions will play an important role in the scope of this work. 
Since it is not used at the moment, we postpone the explicit mapping
for the convenience of the reader. The definition of $\eta_1^j$ and $\eta_2$ will be given in \autoref{def:eta1j2}.
\end{remark}
By an iterative application of \autoref{lemma:kernelsdecreasing}, we obtain that the family of auxiliary kernels has energy forms which are bound from above by the original energy form.
\begin{corollary}\label{kernelsdecreasing}
For every $n\in\N_0$, there is a constant $c>0$ such that for every function $u:\R^d\to\R$,
\[ c \int_{\R^d}\int_{\R^d} (u(x)-u(y))^2\, K^{n}(x,y)\, \d y\, \d x \leq \int_{\R^d}\int_{\R^d} (u(x)-u(y))^2\, K(x,y)\, \d y\, \d x. \]
\end{corollary}
Given the sequence of kernels $K^j$, we can define the corresponding sets of non-degeneracy.
Let us denote the $\sigma$-Algebra of all Lebesgue measurable sets by $\mathcal{M}$.
\begin{definition} Let $a_j>0$ be a given sequence. We define for $j\geq 0$
 \[\mathcal{N}^j:\R^d\to\mathcal{M}, \quad \mathcal{N}^j(x):=\{v\in\R^d\colon K^j(x,v)\geq a_j |x-v|^{-d-2s}\}. \]
 \end{definition}
 \begin{remark}
The sequence $a_j$ will be chosen to be of the form $a_j=c^j\lambda$ for some $c\in(0,1]$ which will be determined in \autoref{rhonond}. 
In particular, $a_j$ is a decreasing sequence of positive real numbers starting at $a_0=\lambda$. 
This means that $\mathcal{N}^0(x) = \{v\in\R^d\colon K(x,v)\geq \lambda |x-v|^{-d-2s}\}$ for $x\in\R^d$.
\end{remark}
 \begin{lemma}\label{sets}
Assume there exist $\mu\in(0,1)$ and $\lambda>0$ such that $K$ satisfies \autoref{assumption}. 
Let $x,z\in\R^d$ and $\delta<\mu/2$.
If there is $A\subset\R^d$ and a ball $B$ such that 
\[ |A\cap B| \geq (1-\delta)|B|,\]
then there exists $\epsilon_0\in(0,1]$, depending on $\mu$, $\delta$ and $d$ only, such that every for $y\in (1+\epsilon_0)B$:
\[ |A\cap (1+\epsilon_0)B\cap \mathcal{N}^0(y)| \geq  \frac{\mu}{2} |(1+\epsilon_0)B|.\]
\end{lemma}
\begin{proof}
Let $\mu\in(0,1)$ and $\lambda>0$ such that $K$ satisfies \autoref{assumption}.
Furthermore, let $x,z\in\R^d$, $\delta<\mu/2$
and
\[\epsilon_0\leq \left(1\wedge \left(\frac{1-\delta}{2-\frac{\mu}{2}}\right)^{1/d}-1\right).\]
Then $\left(1+\epsilon_0\right)^{-d}(1-\delta) \geq 1-\frac{\mu}{2}$
and therefore
\begin{equation}\label{Crhox}
 |A\cap (1+\epsilon_0)B| \geq (1-\delta)|B| = (1-\delta)(1+\epsilon_0)^{-d}|(1+\epsilon_0)B| \geq \left(1-\frac{\mu}{2}\right)|(1+\epsilon_0)B|,  
\end{equation}
By \autoref{assumption}, we conclude for $y\in (1+\epsilon_0)B$
\begin{equation}\label{Crhoy}
|\mathcal{N}^0(y)\cap (1+\epsilon_0)B|\geq \mu|(1+\epsilon_0)B|.
\end{equation}
Combining \eqref{Crhox} and \eqref{Crhoy},
\begin{align*}
 |A\cap (1+\epsilon_0)B\cap \mathcal{N}^0(y)| \geq
\left(1-\frac{\mu}{2}+\mu-1\right) |(1+\epsilon_0)B| = \frac{\mu}{2} |(1+\epsilon_0)B|.
\end{align*}
\end{proof}
 In the following, we specify the functions $\eta_1^j$ and $\eta_2$, which play an important role in the already defined auxiliary kernels $K^j$.
 Before we define $\eta_1^j,\eta_2$, we first give the following definition of auxiliary radii.
 \begin{definition}
Let $j\geq 0$ and $\delta\in(0,1)$, we define $\rho^j_{\delta}:\R^d\times\R^d\to [0,\infty)$,
  \begin{equation}\label{rhos}
\rho^j_{\delta}(x,z) := \sup\{r<\tfrac15 |x-z|\colon  \exists v\in\R^d \, \text{s.t.} \, |\mathcal{N}^j(x)\cap B_{r}(v)| \geq (1-\delta)|B_{r}|\, \text{ and }\,  z\in B_r(v)\}.
\end{equation}
We use the convention $\rho^j_{\delta}(x,z)=0$, whenever the set of radii in \eqref{rhos} is empty.
 \end{definition}
We can now define the functions $\eta_1^j,\eta_2$, which already appeared in \autoref{def:auxkernels} and assumed to satisfy \eqref{eta12}. 
\begin{definition}\label{def:eta1j2}
Let $j\geq 0$ and $\delta\in(0,1)$. We define $\eta_1^j,\eta_2:\R^d\times\R^d\times\R^d\to[0,\infty]$,
 \begin{align*}
 \eta_1^j(x,y,z)&:= \begin{cases} \frac{c_a}{(\rho_{\delta}^{j}(x,z))^d}\mathds{1}_{B_{4\rho_{\delta}^{j}(x,z)}(z)}(y), \quad & \text{if } \rho_{\delta}^{j}(x,z)>0, \\
 0, & \text{if } \rho_{\delta}^{j}(x,z)=0,  \end{cases}\\
  \eta_2(x,y,z) &= c_b|y-z|^{2s}\max(|x-z|,|y-z|)^{-d-2s},
  \end{align*}
where $c_a,c_b>0$ are constants, depending on the dimension $d$ only, such that \eqref{eta12} is satisfied.
\end{definition}
From now on, we assume $\eta_1^j,\eta_2$ to be defined as in \autoref{def:eta1j2}. 
The function $\eta_1^j$ localizes the area of integration in the definition of the auxiliary kernel $K^{j+1}$ as follows:
\begin{lemma}
Let $j\geq 0$. If $x,y\in B$,
\[ K^{j+1}(x,y) = \int_{5B} \min ( K^j(x,z)\eta^j_1(x,y,z),K(y,z)\eta_2(x,y,z)) \, \d z. \]
\end{lemma} 
\begin{proof}
By definition, $\eta_1^j(x,y,z)>0$, iff $|y-z|< 4 \rho_{\delta}^{j}(x,z)$. Note that 
$4 \rho_{\delta}^{j}(x,z)<\frac{4}{5}|x-z|<4|x-y|$ and therefore $\eta_1^j(x,y,z)=0$, whenever $z\notin 5B$. Hence,
\[\int_{(5B)^c} \min ( K^j(x,z)\eta^j_1(x,y,z),K(y,z)\eta_2(x,y,z)) \, \d z = 0. \]
\end{proof}
\begin{corollary}\label{kernelsdecreasinglocal}
For every $n\in\N_0$, there is a constant $c>0$ such that for every function $u:\R^d\to\R$,
\[ \int_{5^nB}\int_{5^nB} (u(x)-u(y))^2\, K(x,y)\, \d y\, \d x \geq c\int_{B}\int_{B} (u(x)-u(y))^2\, K^{n}(x,y)\, \d y\, \d x. \]
\end{corollary}
\section{Growing sets of non-degeneracy}
\label{s:nondeg-sets}
In this section we take a closer look at the previously defined auxiliary sets of non-degeneracy and prove important properties for those objects. 
This section is divided into two parts. In the first part, we prove that there is a sequence $a_j>0$ such that the sets of non-degeneracy $\mathcal{N}^j$ are nested. 
In the second part, we prove a growing ink-spot theorem, 
which gives us a qualitative statement regarding the growth behavior of two consecutive sets. 
\subsection{Nested sets of non-degeneracy}
Recall that for any $x\in\R^d$, the family $\mathcal{N}^j(x)$ is determined by a decreasing sequence of real numbers $a_j>0$ with $a_0=\lambda$ as follows:
\[ \mathcal{N}^j(x):=\{v\in\R^d\colon K^j(x,v)\geq a_j |x-v|^{-d-2s}\}. \]
This subsection aims to prove the existence of such sequence $a_j$ which implies that the sets $\mathcal{N}^j(x)$ 
are nested. The goal of this subsection is to prove the following proposition:
\begin{proposition}\label{prop:nondegnested}
Assume there exist $\mu\in(0,1)$ and $\lambda>0$ such that $K$ satisfies \autoref{assumption}. 
There is a constant $c\in(0,1]$, depending on the dimension $d$ and $\mu$ only, such that the sequence $a_j=c^j\lambda$ satisfies for all $j\in\N_0$ and $x\in\R^d$
\[\mathcal{N}^{j}(x)\subset \mathcal{N}^{j+1}(x)\]
except a set of measure zero.
\end{proposition}
Before proving \autoref{prop:nondegnested}, we first need to prove an auxiliary result, which is the main ingredient in the proof of \autoref{prop:nondegnested}.
\begin{lemma}\label{rhonond}
Assume there exist $\mu\in(0,1)$ and $\lambda>0$ such that $K$ satisfies \autoref{assumption}. 
Let $j\geq 0$ and $a_j\in(0,\lambda]$ be given. 
If $\delta< \mu/2$, there is a constant $c\in(0,1]$, depending on the dimension $d$, $\mu$ and $\delta$ only, such that $a_{j+1}=c\cdot a_j$ satisfies for all $x\in\R^d$
\[ \{ v\in\R^d\colon \rho^j_{\delta}(x,v)>0\}\subset \mathcal{N}^{j+1}(x). \]
\end{lemma}
\begin{proof}
Let $\mu\in(0,1)$ and $\lambda>0$ such that $K$ satisfies \autoref{assumption}. Let $x\in\R^d$, $j\geq 0$ and assume $\delta< \mu/2$.  

Let $y \in\{ v\in\R^d\colon \rho^j_{\delta}(x,v)>0\}$ for a given $a_j>0$. 
The aim is to show that there is a $c>0$, such that $y\in \mathcal{N}^{j+1}(x)$ for $a_{j+1}=c\cdot a_j$, i.e.
\begin{equation}\label{aim:kj+1lemma}
K^{j+1}(x,y) \geq c\cdot a_{j} |x-y|^{-d-2s}.
\end{equation} 
Recall the definition of $K^{j+1}(x,y)$
\[K^{j+1}(x,y) : = \int_{\R^d} \min ( K^j(x,z)\eta^j_1(x,y,z),K(y,z)\eta_2(x,y,z)) \, \d z\]
and note that $\eta_1(x,y,z)>0$, iff 
\begin{equation}\label{yzkappa}
z\in \Omega_j(x,y) := \{z\colon |y-z|<4\rho^j_{\delta}(x,z)\}.
\end{equation} 
Hence, we can reduce the area of integration for $K^{j+1}$ to $\Omega_j(x,y)$.
Since we assumed $\rho^j_{\delta}(x,y)>0$, there is a neighborhood of $y$ in $\Omega_j(x,y)$ and therefore $\Omega_j(x,y)$ is not empty.

Let $x,y$ be as above and $z\in\Omega_j(x,y)$. By positioning of the points, we can uniformly bound the distance $|x-z|$ from above by the distance $|x-y|$.
The triangle inequality implies
$ |x-z|\leq |x-y|+|y-z| <|x-y|+\tfrac{4}{5}|x-z|$, where we used $z\in\Omega_j(x,y)$ in the last inequality. Consequently,
\begin{equation}\label{estimatediff}
|x-z|\leq 5|x-y|.
\end{equation}

We aim to prove that there is a pair $(\widetilde{z},\widetilde{v})\in\R^d\times\R^d$ with $\widetilde{z}\in\Omega_j(x,y)$, such that
\begin{align}
&\rho_{\delta}^j(x,\widetilde{z})\geq \widetilde{c}\rho_{\delta}^j(x,z) \quad \text{ for all } z\in  (1+\epsilon_0)B_{\rho^j_{\delta}(x,\widetilde{z})}(\widetilde{v}), \label{rhodom} \\
&|(1+\epsilon_0)B_{\rho^j_{\delta}(x,\widetilde{z})}(\widetilde{v}) \cap \mathcal{N}^j(x)\cap \mathcal{N}^0(y)|\geq \frac{\mu}{2}|B_{\rho^j_{\delta}(x,\widetilde{z})}(\widetilde{v})| \label{aimball}
\end{align}
for some $\widetilde{c},\epsilon_0>0$, depending on $d$, $\mu$ and $\delta$ only.
This assertion will allow us to reduce the area of integration for $K^{j+1}$ to the favorable area on which we can use the lower bounds for the kernels and the upper bound for $\rho_{\delta}^j(x,z)$ to prove the lemma. 

We define inductively a sequence of points $z_0,z_1,\dots,z_n\in\Omega_j(x,y)$ and $v_0,\dots,v_n\in\R^d$, using a chain argument, such that we can assign for 
each pair $(z_j,v_j)$ a ball $B_{\rho^j_{\delta}(x,z_j)}(v_j)$ with a sufficiently large area of non-degeneracy and such that the radius of the subsequent ball increases at least with a given factor. The sequence will be constructed in such a way that we can apply \autoref{sets} for the last ball $B_{\rho^j_{\delta}(x,z_n)}(v_n)$, which will then imply \eqref{aimball} for the pair $(\widetilde{z},\widetilde{v})=(z_n,v_n).$
As in the proof of \autoref{sets}, let
\[\epsilon_0 < \left(1\wedge \left(\frac{1-\delta}{2-\frac{\mu}{2}}\right)^{1/d}-1\right)\]
and define $\xi = \frac{\epsilon_0+2}{\epsilon_0}$.
The quantity $\xi$ will describe the growth factor for the sequence of balls and $\epsilon_0$ the enlargement of the last ball satisfying \eqref{aimball}.
Note that $\xi>3$, since $\epsilon_0<1$. We construct the sequence of pairs $(z_j,v_j), j\in\{0,\dots,n\}$ as follows:
\begin{enumerate} 
\item[(0)] Set $z_0:=y$. 
Since $\rho^j_{\delta}(x,z_0)>0$, there is $v_0\in\R^d$ such that 
\[z_0\in B_{\rho^j_{\delta}(x,z_0)}(v_0) \quad \text{and} \quad 
 |\mathcal{N}^j(x)\cap B_{\rho^j_{\delta}(x,z_0)}(v_0)|\geq (1-\delta) |B_{\rho^j_{\delta}(x,z_0)}|.\] 
 \item[$(i)$] If there is $z_i\in B_{\rho^j_{\delta}(x,z_{i-1})}(v_{i-1})$ with $\xi \rho^j_{\delta}(x,z_{i-1})<\rho^j_{\delta}(x,z_i)$, choose such $z_i$.
 By the definition of $\rho^j_{\delta}(x,z_i)$, there is  $v_i\in\R^d$ such that 
\[z_i\in B_{\rho^j_{\delta}(x,z_i)}(v_i)\quad \text{and} \quad  |\mathcal{N}^j(x)\cap B_{\rho^j_{\delta}(x,z_i)}(v_i)|\geq (1-\delta) |B_{\rho^j_{\delta}(x,z_i)}|.  \]
\end{enumerate}

\begin{figure}[htb]
\includegraphics[width=\textwidth]{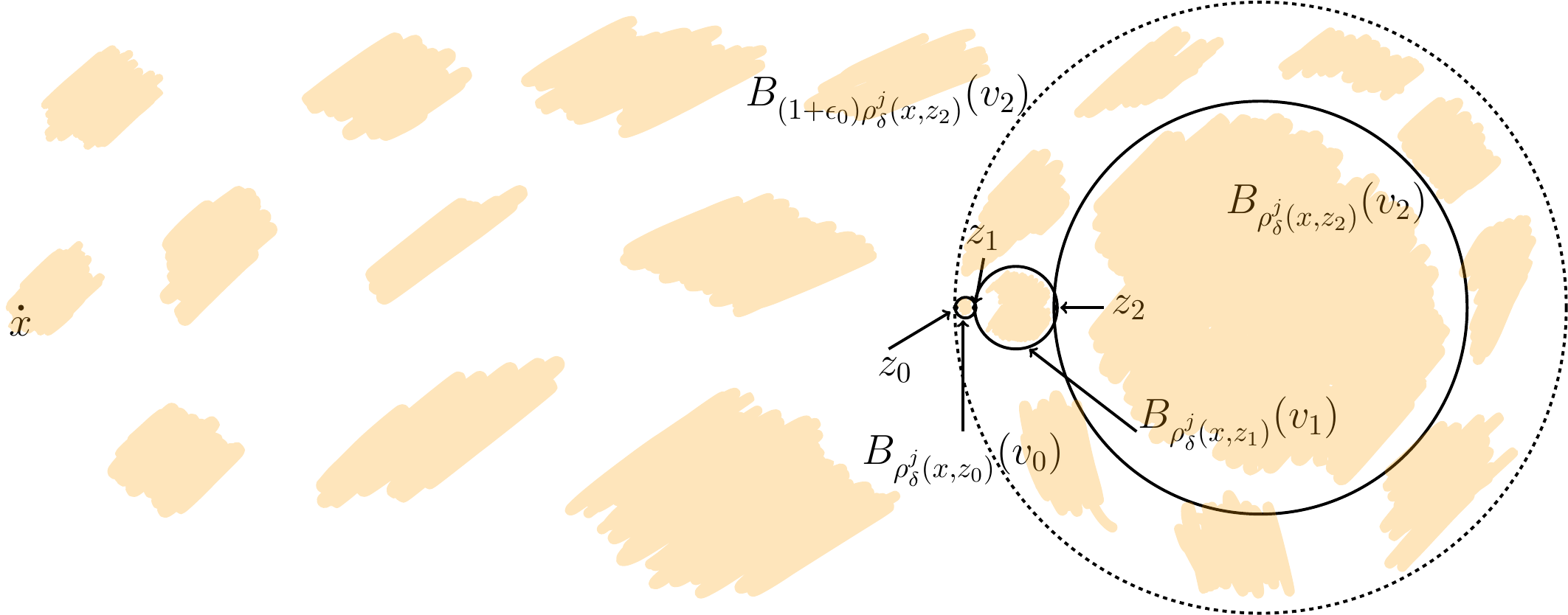}
\caption{Illustration of the points $z_0,z_1,z_2$ and the corresponding balls.}
\end{figure}

The radii $\rho^j_{\delta}(x,z_i)$ grow at least by the factor $\xi$ and $\rho_{\delta}^j(x,z_i)<\frac{1}{5}|x-z_i|$. Hence, the iteration stops after finitely many steps. 
Note that $z_i\in\Omega_j(x,y)$ for all $i\in\{0,\dots,n\}$, since
\[ |y-z_i|\leq |y-v_i|+|v_i-z_i| < \left(2\sum_{k=1}^{\infty} \xi^{-k}+2\right)\rho_{\delta}^j(x,z_i) = \frac{2\xi}{\xi-1}\rho_{\delta}^j(x,z_i) < 4\rho_{\delta}^j(x,z_i). \]

In order to apply \autoref{sets} for $A=\mathcal{N}^j(x)$ and $B=B_{\rho^j_{\delta}(x,z_n)}(v_n)$, it remains to show that $y\in (1+\epsilon_0)B$. 
By construction, 
\begin{equation}\label{distyvn}
 |y-v_n| \leq |z_n-v_n| + |y-z_n| < \left(1+2\sum_{k=1}^{\infty} \xi^{-k}\right)\rho^j_{\delta}(x,z_n) = \left(1+\epsilon_0\right)\rho^j_{\delta}(x,z_n), 
\end{equation}
i.e. $y\in (1+\epsilon_0)B$. Hence, by \autoref{sets}
\begin{equation}\label{enlargeepsilonintersect}
|B_{(1+\epsilon_0)\rho^j_{\delta}(x,z_n)}(v_n) \cap \mathcal{N}^j(x)\cap \mathcal{N}^0(y)| \geq \frac{\mu}{2}|B_{(1+\epsilon_0)\rho^j_{\delta}(x,z_n)}|
\end{equation} 
which proves \eqref{aimball}. 

We can describe the support of $\Omega_j(x,y)$ in terms of the ball $B_{\rho^j_{\delta}(x,z_n)}(v_n)$. To be more precise, by \eqref{distyvn} we deduce
$\Omega_j(x,y)\subset 3(1+\epsilon_0)B_{\rho^j_{\delta}(x,z_n)}(v_n)$.

The sequence $z_i$ is build such that $\xi \rho_{\delta}^j(x,z_n)\geq \rho_{\delta}^j(x,z)$ for all $z\in B_{\rho^j_{\delta}(x,z_n)}(v_n)$. 
Choosing $\epsilon_0$ sufficiently small, proves \eqref{rhodom}.

To simplify notation, let 
$\Xi:= B_{(1+\epsilon_0)\rho^j_{\delta}(x,z_n)}(v_n) \cap \mathcal{N}^j(x)\cap \mathcal{N}^0(y)$. Then by
\eqref{estimatediff}, $a_j\leq \lambda$, $y\in B_{(1+\epsilon_0)\rho^j_{\delta}(x,z_n)}(v_n)$ and \eqref{enlargeepsilonintersect},
\begin{align*}
& K^{j+1}(x,y) = \int_{\R^d} \min (K^j(x,z)\eta^j_1(x,y,z),K(y,z)\eta_2(x,y,z)) \, \d z \\
& \geq  c_1\int_{\Xi} \min\left(K^j(x,z)(\rho^j_{\delta}(x,z))^{-d},K(y,z)|y-z|^{2s}|x-z|^{-d-2s}\right) \d z \\
& \geq c_2 a_j |x-y|^{-d-2s} |B_{(1+\epsilon_0)\rho^j_{\delta}(x,z_n)}(v_n)\cap \mathcal{N}^j(x)\cap \mathcal{N}^0(y)| (\rho^j_{\delta}(x,z_n))^{-d} \\
& \geq \frac{\mu}{2}c_2a_j|x-y|^{-d-2s} |B_{(1+\epsilon_0)\rho^j_{\delta}(x,z_n)}(v_n)|((1+\epsilon_0)\rho^j_{\delta}(x,z_n))^{-d} \\
& = c_3a_j |x-y|^{-d-2s},
\end{align*}
where the constants $c_1,c_2,c_3>0$ depend only on the dimension $d$, $\delta$ and $\mu$.
\end{proof}
We have all tools to prove \autoref{prop:nondegnested}.
\begin{proof}[Proof of \autoref{prop:nondegnested}]
Let $\lambda>0$ and $\mu\in(0,1)$ such that $K$ satisfies \autoref{assumption}.
Let $j\geq 0$, $x\in\R^d$ and $\delta<\mu/2$. If $y\in \mathcal{N}^{j}(x)$ is a Lebesgue point for some $a_j>0$, then $|\mathcal{N}^{j}(x)\cap B|\geq (1-\delta)|B|$ for any sufficiently small ball with $B$ with $x\in B$. In particular $\rho_{\delta}^j(x,y)>0$. Hence, by \autoref{rhonond} there is a constant $c\in(0,1]$, depending on $d$ and $\mu$, such that $y\in  \mathcal{N}^{j+1}(x)$ for $a_{j+1}=c\cdot a_j.$ Since $c_0=\lambda$ and the constant $c$ is independent of $j$ and $x$ , the proposition follows for the
sequence $a_j=c^{j}\lambda$.
\end{proof}
\subsection{Growing Ink-Spots}
As mentioned in the beginning of the section we intend to prove a result concerning the growth behavior for two consecutive auxiliary sets of non-degeneracy. 
It is a growing ink-spot-type theorem which was originally developed by Krylov and Safonov for elliptic equations in non-divergence form.
Our aim is to show that the fraction of two consecutive sets is bounded from below by some constant strictly larger than one, depending on the dimension $d$ and $\mu$ only. 
\begin{proposition}\label{iterationstep1}
Assume there exist $\mu\in(0,1)$ and $\lambda>0$ such that $K$ satisfies \autoref{assumption}.
There are constants $c_1,c_2>0$, depending on $d$ and $\mu$ only, such that for every ball $B_R(z_0)$ and $x\in\R^d$ with $|x-z_0|=(1+c_1)R$ and every $j\geq 0$, either
\begin{equation}\label{eq:growingink}
B_R(z_0)\subset \mathcal{N}^{j+1}(x) \text{ a.e. }  \quad \text{or}\quad 
\frac{|B_R(z_0)\cap \mathcal{N}^{j+1}(x)|}{|B_R(z_0)\cap \mathcal{N}^j(x)|} \geq \left(1+c_2\right). 
\end{equation}
\end{proposition}
Before we address the proof of \autoref{iterationstep1}, we first need to prove an auxiliary result.
It is an geometric observation,
whose application in the proof of \autoref{iterationstep1} provides the existence of balls with desired properties.
\begin{lemma}\label{nondegsmallerball}
Let $R>0$, $z_0\in\R^d$ and $A$ be a measurable set. 
For any $c_0\in(0,1)$ and $0<\delta<3^{-d}$, if
\begin{equation}\label{lemma:ineqcontainer}
|A\cap B_R(z_0)| \geq (1-\delta) |B_R|,
\end{equation}
then there exists a ball $B_{c_0R}(z)\subset B_R(z_0)$ such that 
\begin{equation}\label{lemma:inequsubball}
|A\cap B_{c_0R}(z)| \geq (1-3^d\delta) |B_{c_0R}|.
\end{equation}
\end{lemma}
\begin{proof}
For any finite covering of $B_R(z_0)$ with balls of radius $c_0R$, the Vitali covering lemma implies the existence of a subcollection of disjoint balls $B^1,\dots, B^N$ with
$B^j\subset B_R(z_0)$ and $B_R(z_0)\subset (3B^1\cup\dots\cup 3B^N)$. 
Note that $|A^c\cap B_R(z_0)| \leq \delta |B_R|$ and $|A^c\cap B_{c_0R}(z)| \leq 3^d\delta |B_{c_0R}|$ are equivalent formulations of \eqref{lemma:ineqcontainer} 
and \eqref{lemma:inequsubball} respectively. We prove the assertion by contradiction. Assume \eqref{lemma:inequsubball} is false, that is
$|A^c\cap B^j| > 3^d\delta |B^j|$ for all $j\in\{1,\dots,N\}$. Hence,
\[ |A^c\cap B_R(z_0)|\geq \sum_{j=1}^N |A^c\cap B^j| >  \sum_{j=1}^N 3^d\delta |B^j| \geq \delta |B_R|.\]
\end{proof}
We finally have all tools to prove the second main result concerning the auxiliary sets of non-degeneracy. 
\begin{proof}[Proof of \autoref{iterationstep1}]
Let $\mu\in(0,1)$ and $\lambda>0$ such that $K$ satisfies \autoref{assumption}.
By \autoref{prop:nondegnested}, there is a constant $c\in(0,1]$ such that the sequence $a_j=c^j\lambda$ satisfies for any $n\geq 0$ and $x\in\R^d$, 
\[\mathcal{N}^0(x)\subset \mathcal{N}^1(x)\subset \mathcal{N}^2(x)\subset \cdots  \subset \mathcal{N}^n(x)\subset\mathcal{N}^{n+1}(x)\]
almost everywhere. Recall that by \autoref{lemm:ass}, \autoref{assumption} is equivalent
to the existence of $\widetilde{\mu}\in(0,1)$ and $c_1>0$, depending only on $d$ and $\mu$, 
such that for every ball $B_R(z_0)$ with $|x-z_0|=(1+c_1)R$:
\[ |\mathcal{N}^0(x)\cap B_R(z_0)|\geq \widetilde{\mu}|B_R|. \]
Let $\delta=\widetilde{\mu}/3^{d+1}$. 

We distinguish between two cases:\\
\textbf{Case 1:} Assume $|B_R(z_0)\cap \mathcal{N}^j(x)|<(1-\delta)|B_R|$.\\
Let $y\in \mathcal{N}^j(x)\cap B_R(z_0)$ be a Lebesgue point and $B^y$ be the largest ball  in $B_R(z_0)$ with $y\in B^y$ and $|\mathcal{N}^j(x)\cap B^y|\geq (1-\delta)|B^y|$. 
Since $B^y$ is chosen to be the largest ball
satisfying $|\mathcal{N}^j(x)\cap B^y|\geq (1-\delta)|B^y|$ and we assumed $|B_R(z_0)\cap \mathcal{N}^j(x)|<(1-\delta)|B_R|$, we conclude by continuity 
\begin{equation}\label{Bydelta}
|B^y\cap\mathcal{N}^j(x)|= (1-\delta)|B^y|.
\end{equation}
Let $r^y$ denote the radius of $B^y$.

We distinguish between three subcases:
\begin{enumerate}
\item Assume $r^y\leq \frac15\dist(x,B^y)$. 
Recall that by \autoref{rhonond}, $y\in \mathcal{N}^{j+1}(x)$ for all $y\in\R^d$ with $\rho_{\delta}^j(x,y)>0$. 
Since $B^y$ satisfies \eqref{Bydelta} and $r^y\leq \frac15\dist(x,B^y)$, we have $\rho_{\delta}^j(x,y)>0$ for all $y\in B^y$ and therefore $B^y\subset \mathcal{N}^{j+1}(x)$. Hence, we obtain
\[|B^y\cap \left(\mathcal{N}^{j+1}(x)\setminus \mathcal{N}^{j}(x)\right)|=\delta |B^y|.\]
\item\label{covcase} Assume $r^y> \frac15\dist(x,B^y)$. In addition, we assume there is a covering for $B^y$ by a family of balls $(B_{i})_{i=1,\dots,N}$ satisfying for all $i\in\{1,\dots,N\}$
\begin{itemize}
\item $B_i$ has radius $\frac{1}{5}\dist(x,B_i)$, 
\item $|B_i\cap\mathcal{N}^j(x)|\geq (1-3^d\delta)|B_i|$.
\end{itemize}  
Using the property $3^d\delta<\widetilde{\mu}/2$ and \autoref{rhonond}, we deduce $B_i\subset \mathcal{N}^{j+1}(x)$ for all $i\in\{1,\dots,N\}$. Therefore,
$B^y\subset \mathcal{N}^{j+1}(x)$ and
\[|B^y\cap \left(\mathcal{N}^{j+1}(x)\setminus \mathcal{N}^{j}(x)\right)|=\delta |B^y|.\]
\item\label{nocovcase} Assume $r^y> \frac15\dist(x,B^y)$ and there is no covering as in \eqref{covcase}. 
In this case we show that there is a small ball inside $B^y$ whose radius is comparable to $r^y$ and for which we can apply \autoref{rhonond}. 

First note that since we assume that there is no covering as in the second subcase, 
we can find a ball $B\subset B_R(z_0)$ with radius $\frac15\dist(x,B)$ and $|B\cap \mathcal{N}^j(x)|<(1-3^d\delta)|B|$.
Applying \autoref{nondegsmallerball} for $A=\mathcal{N}^j(x)$, there is a ball ${B}^{\ast}\subset B^y$ with same radius as $B$ such that $|{B}^{\ast}\cap\mathcal{N}^j(x)|\geq(1-3^d\delta)|{B}^{\ast}|$. 
Hence by continuity, we can find a ball $\widetilde{B}\subset B^y$ with same radius as $B$ and $B^{\ast}$ such that $|\widetilde{B}\cap\mathcal{N}^j(x)|=(1-3^d\delta)|\widetilde{B}|$.
By \autoref{rhonond}, $\widetilde{B}\subset \mathcal{N}^{j+1}(x)$. Since $\widetilde{B}\subset B^y\subset B_R(z_0)$, the radii satisfy
\[\frac{c_1r^y}{5}\leq\frac{c_1R}{5}\leq \frac{1}{5}\dist(x,\widetilde{B}) \leq\frac15\dist(x,B^y)<r^y.\]
We conclude
\[|B^y\cap \left( \mathcal{N}^{j+1}(x)\setminus \mathcal{N}^j(x) \right)|\geq 3^d\delta |\widetilde{B}| \geq \delta\frac{(3c_1)^d}{5^d}|B^y|.\]
\end{enumerate}

\begin{figure}[htb]
\includegraphics[width=0.8\linewidth]{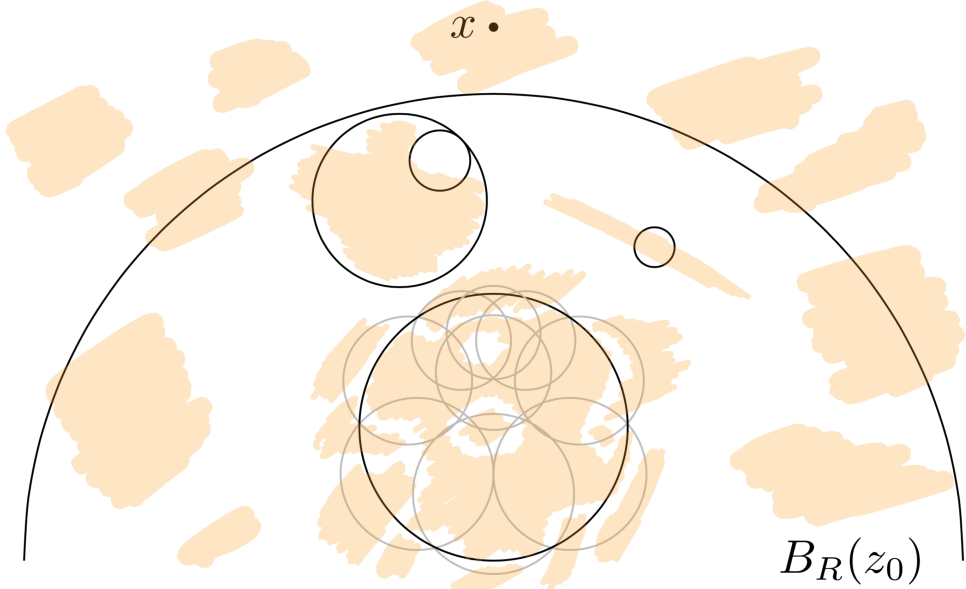}
\caption{The figure shows the ball $B_R(z_0)$ and the set $\mathcal{N}^j(x)\cap B_R(z_0)$. 
The small balls inside $B_R(z_0)$ on the upper right represent the first subcase.
The ball on the upper left represents the third subcase, where we can see a smaller ball $B$ inside $B^y$ with radius 
$\dist(x,B)/5$ and $|B_i\cap \mathcal{N}^j(x)|\geq |1-3^d)\delta |B_i|$.
The central ball represents $B^y$ in the second subcase, which satisfies $r^y= \frac15\dist(x,B^y)>\frac15\dist(x,B^y)$. The gray balls demonstrate a covering $B_i$ of $B^y$ satisfying $\radius(B_i)=\dist(x,B_i)/5$ and $|B_i\cap \mathcal{N}^j(x)|<(1-3^d)\delta |B|$.}
\end{figure}

The family of balls $B^y$ covers $B_R(z_0)\cap \mathcal{N}^j(x)$ almost everywhere.\\
Using the Vitali covering lemma, we can select a finite subcollection of non-overlapping balls $B^j$ such that 
$\left(B_R(z_0)\cap\mathcal{N}^j(x)\right)\subset (3B^1\cup\dots\cup 3B_N)$ expect for a set of measure zero.\\
Altogether,
\begin{align*}
|B_R(z_0) \cap \left(\mathcal{N}^{j+1}(x)\setminus \mathcal{N}^j(x)\right)| 
& \geq  \sum_{j=1}^N |\left(\mathcal{N}^{j+1}(x)\setminus \mathcal{N}^j(x)\right)\cap B^j| \geq \sum_{j=1}^N \delta \left(\frac{3c_1}{5}\right)^d |B^j| \\
& = \delta \left(\frac{c_1}{5}\right)^d  \sum_{j=1}^N |3 B^j| \geq \delta \left(\frac{c_1}{5}\right)^d |B_{R}(z_0)\cap\mathcal{N}^j(x)|.
\end{align*}
Hence there is $c_3>0$, depending on $d$ and $\mu$, such that
\begin{align*} 
|B_{R}(z_0)\cap\mathcal{N}^{j+1}(x)| 
 \geq \left(1+c_3\right) |B_{R}(z_0)\cap \mathcal{N}^j(x)|. 
\end{align*}
\textbf{Case 2:} Assume $|B_R(z_0)\cap \mathcal{N}^j(x)|\geq(1-\delta)|B_R|$.\\
In this case we do not cover $B_R(z_0)\cap \mathcal{N}^j(x)$ by a family of balls and consider directly $B_R(z_0)$.
We make a distinction between the following two subcases:
\begin{enumerate} \setcounter{enumi}{3}
\item If there exists a covering of $B_R(z_0)$ as in \eqref{covcase}, then we conclude with the same argument as in \eqref{covcase} and conclude
$B_{R}(z_0)\subset \mathcal{N}^{j+1}(x)$.
\item If there is no covering of $B_R(z_0)$ as in \eqref{covcase}, then we proceed as in \eqref{nocovcase}. 

In this case, there is a ball $B\subset B_R(z_0)$ with radius $\frac15\dist(x,B)$ such that $|B\cap\mathcal{N}^j(x)|=(1-3^d\delta)|B|$
and $B\subset \mathcal{N}^{j+1}(x)$. Hence, 
\[|B_R(z_0)\cap \left( \mathcal{N}^{j+1}(x)\setminus \mathcal{N}^j(x) \right)|\geq 3^d\delta |B| =c_4|B_R|\]
for some $c_4>0$, depending on $d$ and $\mu$. \\
Proceeding as in Case 1, finishes the proof.
\end{enumerate}
\end{proof}
An immediate consequence of \autoref{iterationstep1} is the following corollary. 
It gives us an upper bound for the amount of steps we need until the set of non-degeneracy fills up the whole space.
It is important to emphasize that the amount of steps does only depend on $\mu$ and $d$.
\begin{corollary}\label{cor:iteration2}
Assume there exist $\mu\in(0,1)$ and $\lambda>0$ such that $K$ satisfies \autoref{assumption}.
There is $n_0\in\N$, depending only on $\mu$ and $d$, such that for every $n\geq n_0$ and $x\in\R^d$,
\[\mathcal{N}^{n}(x)=\R^d \text { a.e.}.\]
\end{corollary}
\begin{proof}
Let $x\in\R^d$. By \autoref{iterationstep1}, there are constants $c_1,c_2>0$, depending on $d$ and $\mu$ only, such that \eqref{eq:growingink}
holds for all balls $B_R(z_0)$ with $|x-z_0|=(1+c_1)R$. 
Choosing $n_0\geq \log(\mu^{-1})/\log(1+c_2)$ implies $B_R(z_0)\subset \mathcal{N}^n(x)$ a.e.. 
Since the choice of $n_0$ is independent of $R$ and $z_0$, we conclude $\mathcal{N}^{n}(x)=\R^d$ except for a set of measure zero.
\end{proof}
\section{Proofs of the main results}
\label{s:theorems}
In this section we prove the coercivity estimates \autoref{thm:coercive} and \autoref{thm:local}. We have already proven all tools we need to deduce those results. 
\autoref{thm:coercive} is an immediate consequence of \autoref{kernelsdecreasing} and \autoref{cor:iteration2}. The proof of \autoref{thm:local} needs some additional work. 
For the sake of clarity, we will separate parts of its proof into lone results, see \autoref{subseclocal}.
\subsection{Proof of \autoref{thm:coercive}}
\begin{proof}
Let $\mu\in(0,1)$ and $\lambda>0$ be such that $K$ satisfies \autoref{assumption}.
By \autoref{cor:iteration2}, there is $n\in\N$, depending on $d$ and $\mu$, such that for every $x\in\R^d$, $\mathcal{N}^{n}(x)=\R^d$ a.e..
Thus $K^n(x,y)\geq a_n|x-y|^{-d-2s}$ for almost every pair $(x,y)\in\R^d\times\R^d$. Hence, by \autoref{kernelsdecreasing} there is a constant $c_1>0$ depending on $n$, such that
\begin{align*} 
 \int_{\R^d}\int_{\R^d} (u(x)-u(y))^2 K(x,y)\, \d y\, \d x & \geq c_1 \int_{\R^d}\int_{\R^d} (u(x)-u(y))^2 K^n(x,y)\, \d y\, \d x \\
 & \geq  c_1\cdot a_n\int_{\R^d}\int_{\R^d} (u(x)-u(y))^2 |x-y|^{-d-2s}\, \d y\, \d x \\
 & = c_1\cdot a_n\|u\|^2_{\dot{H}^s(\R^d)}.
 \end{align*}
 Recall that by \autoref{prop:nondegnested} the sequence $a_n$ is given by $a_n=c^n\lambda$ for some constant $c>0$, depending on $d$ and $\mu$, which finishes the proof.
\end{proof}
\subsection{Proof of \autoref{thm:local}}\label{subseclocal}
In this subsection we prove \autoref{thm:local}. The idea of the proof is to cover $B_1$ by small balls, whose radii depend on the dimension $d$ and 
the value of $\mu$ from \autoref{assumption}. 
We first show that for any given ball, 
there is a scaling factor for the radius such that the local energy form for $K$ on the scaled ball can be bounded from below by the $H^s$-seminorm on the original ball.
\begin{lemma}\label{smallballlemma}
Assume there exist $\lambda>0$ and $\mu\in(0,1)$ such that $K$ satisfies \autoref{assumption}.
There are constants $c>0$ and $n\in\N$, depending on $d$ and $\mu$, such that for every function $u:\R^d\to\R$ and every ball $B\subset\R^d$
\begin{equation}\label{smallballsest} 
 \int_{5^nB}\int_{5^nB} (u(x)-u(y))^2 K(x,y)\, \d y\, \d x \geq c\lambda\|u\|^2_{\dot{H}^s(B)}.
 \end{equation}
\end{lemma}
\begin{proof}
Let $\mu\in(0,1)$ and $\lambda>0$ be such that $K$ satisfies \autoref{assumption}.
Proceeding as in the proof of \autoref{thm:coercive}, by \autoref{cor:iteration2} and \autoref{kernelsdecreasinglocal} there are constants $c_1>0$ and $n\in\N$, depending on $d$ and $\mu$, such that for every ball $B\subset\R^d$ the assertion follows.
\end{proof}
Let $\mathcal{C}$ be a finite covering of $B_1$ with balls $B^j$ satisfying $\radius(B^j) = \frac{1}{3\cdot 5^{n}}$ and $\cent(B^j)\in B_1$. 
Since $\mathcal{C}$ consists of balls with same radius, such covering $\mathcal{C}$ can be chosen such that $|\mathcal{C}|$ depends on the radius of those balls and the dimension only. A rough covering of a cube with side length $2$ by such balls can be chosen with less then $(2+6\cdot 5^n)^d$ balls and therefore $B_1$ can be covered by less then $(2+6\cdot 5^n)^d$ balls. The radius of the covering balls is chosen so small such that for every covering ball the $3\cdot 5^n$-scaled ball remains inside $B_2$. 
\begin{figure}[htb]
\includegraphics[scale=0.16]{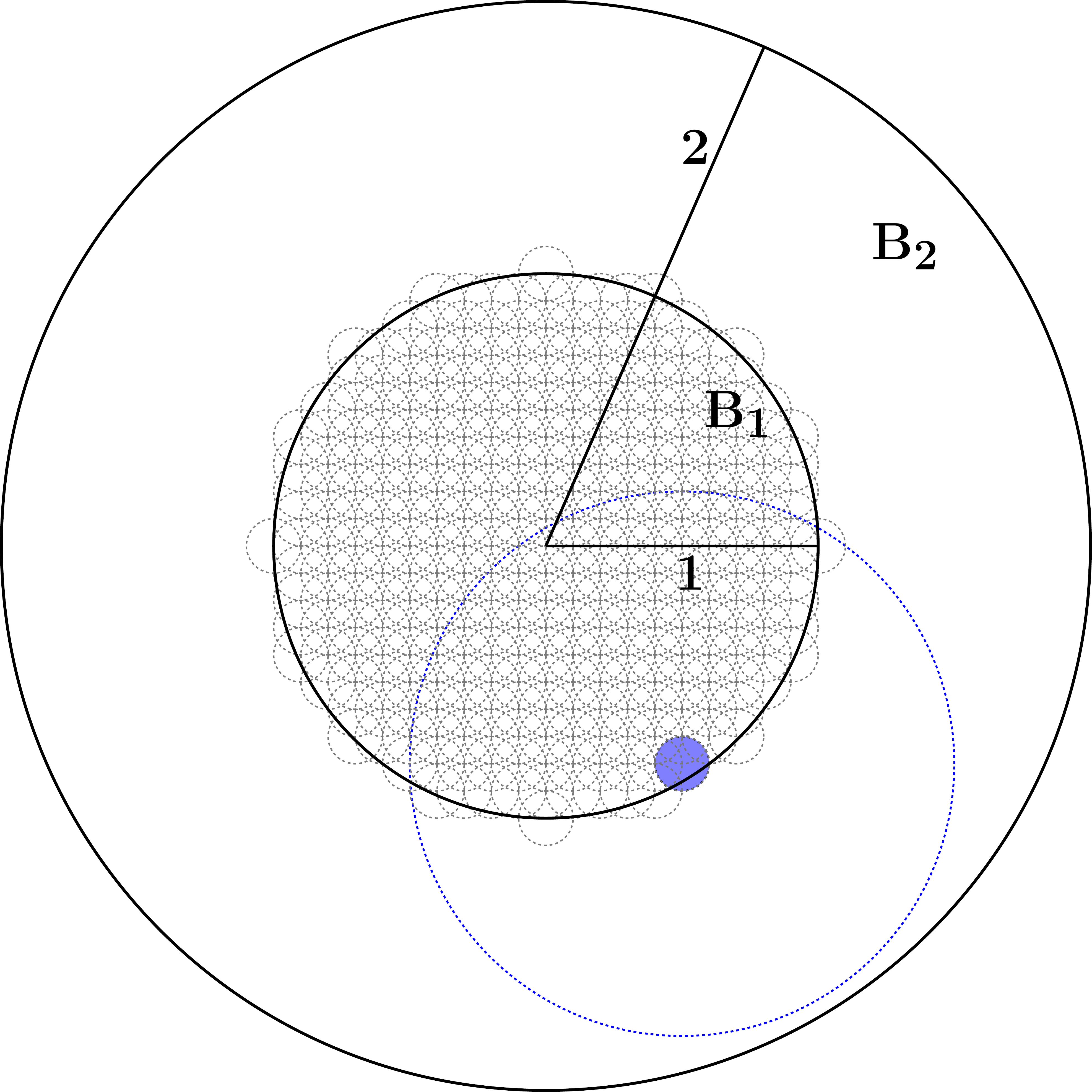}
\caption{Illustration of a rough covering of $B_1$ by small balls and an example of a covering ball and the $3\cdot 5^n$ scaling of this ball.}
\end{figure}

\begin{proposition}\label{propBkBl}
Assume there exist $\lambda>0$ and $\mu\in(0,1)$ such that $K$ satisfies \autoref{assumption}. 	
Let $B^k, B^l$ be two balls with $\cent(B^k), \cent(B^l) \in B_1$ and $\radius(B^k)=\radius(B^l)=\frac{1}{3\cdot 5^{n}}$.
There is a constant $c>0$, depending on $d$ and $\mu$, such that for every function $u:\R^d\to\R$
\begin{equation}\label{B2B2BkBl}
\int_{B_2}\int_{B_2} (u(x)-u(y))^2 K(x,y)\, \d y\, \d x \geq c\lambda \int_{B^k}\int_{B^l} (u(x)-u(y))^2 |x-y|^{-d-2s}\, \d y\, \d x.
\end{equation} 
\end{proposition}
\begin{proof}
Let $\mu\in(0,1)$ and $\lambda>0$ be such that $K$ satisfies \autoref{assumption}.
By definition of the balls $B^k, B^l$, we have $5^nB^k, 5^nB^l\subset B_2$. 
In the following, we investigate three cases which relate to the positioning of the balls $B^k, B^k$.

\begin{enumerate} 
\item If $B^k=B^l$, then the assertion is an immediate consequence of \autoref{smallballlemma} and the observation $5^nB^k, 5^nB^l\subset B_2$.
\item\label{Case2BkBl} Let $B^k\neq B^l$ with $\dist(B^k, B^l)\leq \radius(B^k)$. In this case, we can cover the balls by a larger ball and again use \autoref{smallballlemma}. 

To be more precise,
we replace the area of integration $B^k\times B^l$ on the right-hand side of \eqref{B2B2BkBl} by 
$\widetilde{B}\times\widetilde{B}$ for some ball $\widetilde{B}$ with 
$\radius(\widetilde{B})=3\radius(B^k)$ and $\cent(\widetilde{B})\in B_1$ satisfying $B^k, B^l\subset \widetilde{B}$. 
Since $5^n\radius(\widetilde{B}) = 3\cdot 5^n \radius(B^k) = 1$, we have $5^n\widetilde{B}\subset B_2$ and therefore the assertion again follows by 
\autoref{smallballlemma}.

\begin{figure}[htb]
\includegraphics[scale=0.15]{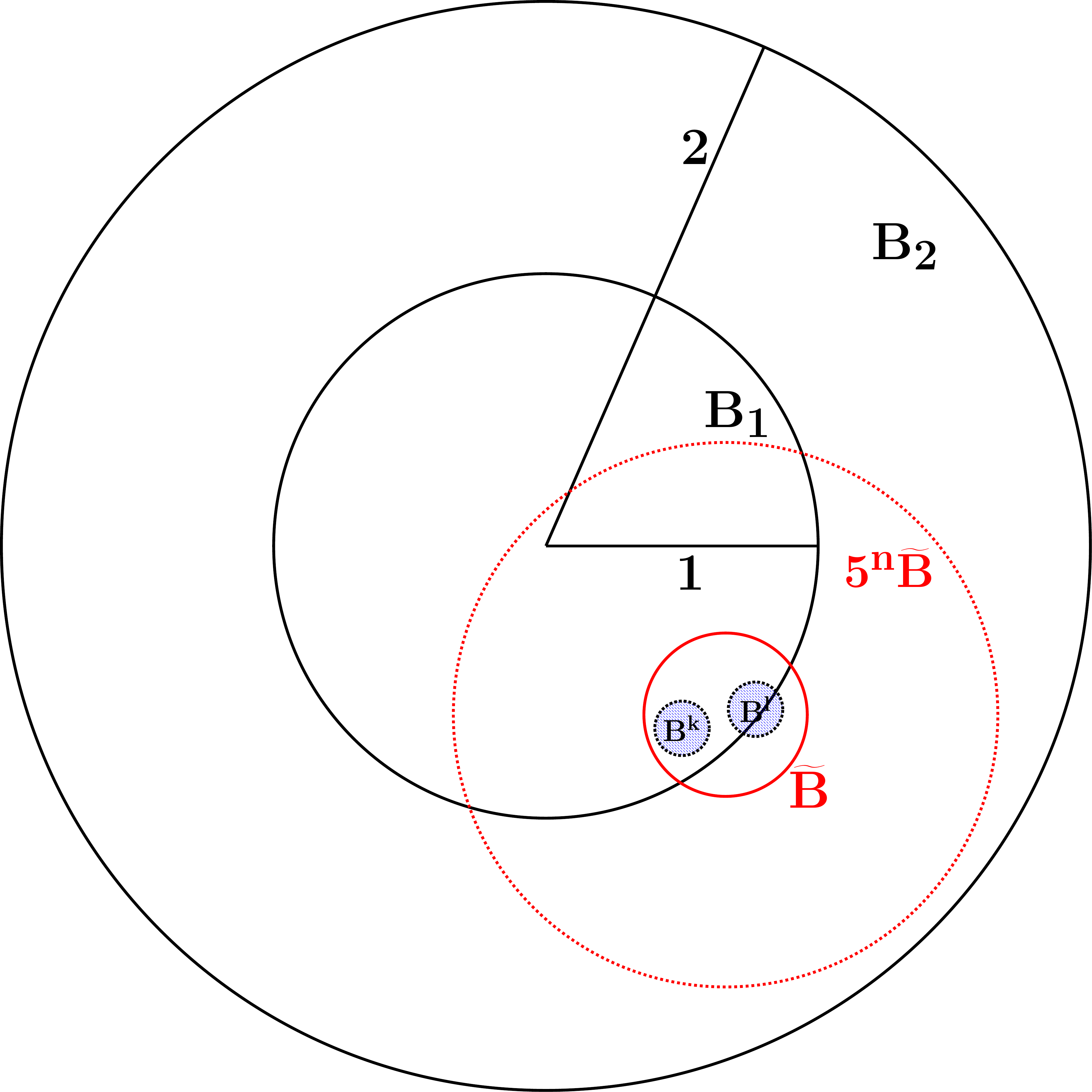}
\caption{The figure illustrates Case \eqref{Case2BkBl}. It shows an example of balls $B^k, B^l$ with $\dist(B^k, B^l)\leq \radius(B^k)$. The ball $\widetilde{B}$ contains these two balls, has its center in $B_1$ and triple radius. Its scaling $5^n\widetilde{B}$ is contained in 
$B_2$.}
\end{figure}

\item Let $B^k\neq B^l$ with $\dist(B^k, B^l)\geq \radius(B^k)$. 
In this case, the idea is to define a sequence of balls such that two consecutive balls intersect and we can estimate stepwise the corresponding double integrals.

We define a sequence of connecting balls 
$B^{k,l}_j$, $j\in\{1,\dots,N\}$ such that for every $j$ 
\begin{itemize}
\item $\radius(B^{k,l}_j)= \radius(B^k)$,
\item $\cent(B^{k,l}_j)\in \{(1-t)\cent(B^k) + t\cent(B^l)\colon t\in[0,1]\}$,
\item $|B^k\cap B^{k,l}_1|=\frac{1}{10}|B^k|=|B^{k,l}_{j+1}\cap B^{k,l}_j| = \frac{1}{10}|B^{k,l}_j|$
\item $|B^l\cap B^{k,l}_N|\geq\frac{1}{10}|B^l|$.
\end{itemize}

\begin{figure}[htb]
\includegraphics[scale=0.5]{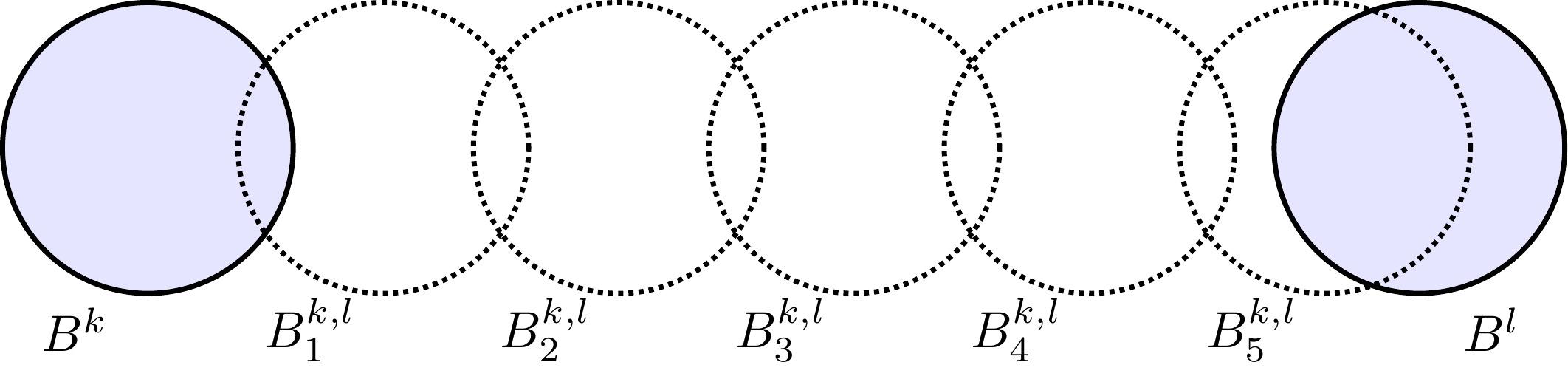}
\caption{Illustration of the balls $B^k, B^l$ with $\dist(B^k, B^l)\geq \radius(B^k)$ and the sequence $B^{k,l}_j$.}
\end{figure}
{\ } \\
Since $B^k, B^l\in \mathcal{C}$ and $\radius(B^k)=\frac{1}{3\cdot 5^n}$, we easily see $N\leq 2+6\cdot 5^n$. 
Hence, $N$ is bounded by a constant depending on $d$ and $\mu$ only.
We distinguish between the cases $N=1$ and $N\geq 2$. In the case $N=1$, we have
\begin{align*}
 \int_{B^k}\int_{B^l} & (u(x)-u(y))^2 |x-y|^{-d-2s}\, \d y\, \d x \\
& \leq \frac{100}{|B^k|^2}\int_{B^k}\int_{B^l} \int_{B^k\cap B^{k,l}_1}\int_{B^l\cap B^{k,l}_1} (u(x)-u(y))^2 |x-y|^{-d-2s} \,\d z_2\,\d z_1 \,  \d y\, \d x  \\
& \leq \frac{300}{|B^k|^2}\Big(\int_{B^k}\int_{B^l} \int_{B^k\cap B^{k,l}_1}\int_{B^l\cap B^{k,l}_1} (u(x)-u(z_1))^2 |x-y|^{-d-2s} \,\d z_2\,\d z_1 \,  \d y\, \d x  \\
& \qquad+  \int_{B^k}\int_{B^l} \int_{B^k\cap B^{k,l}_1}\int_{B^l\cap B^{k,l}_1} (u(z_1)-u(z_2))^2 |x-y|^{-d-2s} \,\d z_2\,\d z_1 \,  \d y\, \d x  \\
& \qquad+  \int_{B^k}\int_{B^l} \int_{B^k\cap B^{k,l}_1}\int_{B^l\cap B^{k,l}_1} (u(z_2)-u(y))^2 |x-y|^{-d-2s} \,\d z_2\,\d z_1 \,  \d y\, \d x\Big) \\
& =:\frac{300}{|B^k|^2}(A_1+A_2+A_3).
\end{align*}
The terms $A_1$ and $A_3$ can be estimated in the same spirit and therefore, we just investigate $A_1$ and $A_2$. 
Note, $|x-y|>|x-z_1|$ for all $x\in B^k, y\in B^l$, $z_1\in B^k\cap B^{k,l}_1$. By \autoref{smallballlemma},
\begin{align*}
A_1 &\leq \int_{B^k}\int_{B^l} \int_{B^k\cap B^{k,l}_1}\int_{B^l\cap B^{k,l}_1} (u(x)-u(z_1))^2 |x-z_1|^{-d-2s} \,\d z_2\,\d z_1 \,  \d y\, \d x \\
& \leq |B^k|^2\int_{B^k} \int_{B^k\cap B^{k,l}_1} (u(x)-u(z_1))^2 |x-z_1|^{-d-2s}\,\d z_1 \, \d x \\ 
& \leq |B^k|^2\int_{B^k} \int_{B^k} (u(x)-u(z_1))^2 |x-z_1|^{-d-2s}\,\d z_1 \, \d x \\ 
& \leq \frac{|B^k|^2}{c_1\lambda} \int_{B_2}\int_{B_2} (u(x)-u(y))^2 K(x,y)\, \d y\, \d x,
\end{align*}
for some constant $c_1>0$, depending on $d$ and $\mu$.
It remains to estimate $A_2$. 
Since $\dist(B^k,B^l)\geq \radius(B^k)$, we obtain $|x-y|\geq \frac{1}{2}|z_1-z_2|$ for all $x\in B^k, y\in B^l$, $z_1\in B^k\cap B^{k,l}_1$ and $z_2\in B^l\cap B^{k,l}_1$. Hence
\begin{align*}
A_2 &\leq 2^{d+4}\int_{B^k}\int_{B^l} \int_{B^k\cap B^{k,l}_1}\int_{B^l\cap B^{k,l}_1} (u(z_1)-u(z_2))^2 |z_1-z_2|^{-d-2s} \,\d z_2\,\d z_1 \,  \d y\, \d x \\
& \leq |B^k|^2 2^{d+4}\int_{B^{k,l}_1}\int_{B^{k,l}_1} (u(z_1)-u(z_2))^2 |z_1-z_2|^{-d-2s} \,\d z_2\,\d z_1 \\
& \leq \frac{|B^k|^2 }{c_2\lambda} \int_{B_2}\int_{B_2} (u(x)-u(y))^2 K(x,y)\, \d y\, \d x.
\end{align*}
for some constant $c_2>0$, depending on $d$ and $\mu$.
Combining these estimates proves the assertion in this case.

It remains to consider the case $N\geq2$.
To simplify notation, let us rename $x=z_0$ resp. $y=z_{n+2}$ and define $B^{k,l}_0:=B^k$ and $B^{k,l}_{N+1}:=B^l$. 
Let $z_0\in B^k$, $z_{N+2}\in B^l$ and $z_j\in B^{k,l}_{j-1}\cap B^{k,l}_{j}$ for $j\in\{1,\dots,N+1\}$. Since, $N\geq 2$, $|z_0-z_{N+2}|\geq |z_{j-1}-z_{j}|$ for all $j\in\{1,\dots,N+2\}$. Hence by the same idea as in the case $N=1$, we conclude
\begin{equation}
\begin{aligned}
& \int_{B^k}\int_{B^l} (u(z_0)-u(z_{N+2}))^2 |z_0-z_{N+2}|^{-d-2s}\, \d z_0\, \d z_{N+2} \\
& \leq \frac{(N+2)10^{N+1}}{|B^k|^{N+1}} \Big( \int_{B^k}\int_{B^l}\int_{B^k\cap B^{k,l}_1}\left(\prod_{i=1}^{N-1}\int_{B^{k,l}_i\cap B^{k,l}_{i+1}}\right)\int_{B^{k,l}_N\cap B^{l}} \\
& \hspace*{8em} \sum_{j=1}^{N+2} (u(z_{j-1})-u(z_j))^2 |z_0-z_{N+2}|^{-d-2s} \, \d z_{N+1}\, \cdots \d z_{1}\, \d z_{N+2}\, \d z_{0} \Big) \\
& \leq \frac{(N+2)10^{N+1}|B^k|^{N+1}}{10^{N-2}|B^k|^{N+1}}\sum_{j=1}^{N+2} \int_{B^{k,l}_{j-1}}\int_{B^{k,l}_{j-1}} (u(z_{j-1})-u(z_j))^2 |z_{j-1}-z_{j}|^{-d-2s}\, \d z_{j} \, \d z_{j-1} \\
&  \leq \frac{1}{c_3\lambda}\sum_{j=1}^{N+2} \int_{5^nB^{k,l}_{j-1}}\int_{5^nB^{k,l}_{j-1}} (u(z_{j-1})-u(z_j))^2 K(z_{j-1},z_{j})\, \d z_{j} \, \d z_{j-1} \\
& = \frac{1}{c_4\lambda} \int_{B_2}\int_{B_2} (u(x)-u(y))^2 K(x,y)\, \d y \, \d x,
\end{aligned}
\end{equation}
for some $c_3,c_4>0$, depending on $d$ and $\mu$.
\end{enumerate} 
\end{proof}
We can finally prove our second main result.
\begin{proof}[Proof of \autoref{thm:local}]
Let $\mu\in(0,1)$ and $\lambda>0$ be such that $K$ satisfies \autoref{assumption} and 
let $\mathcal{C}$ be a finite covering of $B_1$ with balls $B^j$ satisfying $\radius(B^j) = \frac{1}{3\cdot 5^{n}}$ and $\cent(B^j)\in B_1$. 
Then by \autoref{propBkBl} there is a constant $c>0$, depending on $d$ and $\mu$ such that for all $k,l$
\[\int_{B_2}\int_{B_2} (u(x)-u(y))^2 K(x,y)\, \d y\, \d x \geq c\lambda \int_{B^k}\int_{B^l} (u(x)-u(y))^2 |x-y|^{-d-2s}\, \d y\, \d x. \]
Hence,
\begin{align*}
 \|u\|^2_{\dot{H}^s(B_1)} & \leq \|u\|^2_{\dot{H}^s(\bigcup_{j} B^j)} =  \int_{\bigcup_{j} B^j}\int_{\bigcup_{j} B^j} (u(x)-u(y))^2 |x-y|^{-d-2s}\, \d y\, \d x \\
 & \leq \sum_{k,l} \int_{B^k}\int_{B^l} (u(x)-u(y))^2 |x-y|^{-d-2s}\, \d y\, \d x  \\
 & \leq \frac{|\mathcal{C}|^2}{c\lambda} \int_{B^k}\int_{B^l} (u(x)-u(y))^2 K(x,y)\, \d y\, \d x,
\end{align*}
which proves the assertion.
\end{proof}

\bibliographystyle{abbrv}
\bibliography{lit}
\end{document}